\documentclass[11pt,final]{amsart}
\usepackage[utf8]{inputenc}
\usepackage[T1]{fontenc}
\usepackage[letterpaper,centering]{geometry}
\usepackage{lmodern}
\usepackage{eucal}
\usepackage{mathrsfs}
\usepackage{enumerate}
\usepackage{amsmath,amssymb,amsfonts}
\usepackage{xcolor}
\definecolor{darkblue}{rgb}{0,0,0.3}
\definecolor{darkgreen}{rgb}{0,0.4,0}
\usepackage[colorlinks=true,linkcolor=darkblue, citecolor=darkgreen, unicode,pdfborder={0 0 0}]{hyperref}
\usepackage[ps,arrow,matrix,tips,line, curve]{xy}
{\setbox0\hbox{$ $}}\fontdimen16\textfont2=\fontdimen17\textfont2
\entrymodifiers={+!!<0pt,\the\fontdimen22\textfont2>}
\SelectTips{cm}{11}
\usepackage{enumitem}
\setlist[enumerate]{label={\upshape(\arabic*)},topsep=.7ex, leftmargin=*}
\setlist[itemize]{leftmargin=*}

\usepackage{ifdraft}
\ifdraft{
	\usepackage[notcite,notref]{showkeys}
	\addtolength{\hoffset}{1.5cm}
	
}{}

\usepackage{etoolbox}

\theoremstyle{plain}

\newtheorem{thm}{Theorem}[section]

\newtheorem{conjecture}[thm]{Conjecture}

\newtheorem{lem}[thm]{Lemma}
\newtheorem{lemma}[thm]{Lemma}

\newtheorem{prop}[thm]{Proposition}

\theoremstyle{definition}

\newtheorem{definition}[thm]{Definition}

\newtheorem{example}[thm]{Example}

\numberwithin{equation}{section}

\input cyracc.def
\DeclareFontFamily{U}{russian}{}
\DeclareFontShape{U}{russian}{m}{n}
{ <5><6> wncyr5
	<7><8><9> wncyr7
	<10><10.95><12><14.4><17.28><20.74><24.88> wncyr10 }{}
\DeclareSymbolFont{Russian}{U}{russian}{m}{n}
\DeclareSymbolFontAlphabet{\mathcyr}{Russian}
\makeatletter
\let\@math@cyr\mathcyr
\renewcommand{\mathcyr}[1]{\@math@cyr{\cyracc #1}}
\makeatother
\newcommand{\Sha}{{\mathcyr{Sh}}}

\makeatletter
\def\myrightarrow{{\setbox\z@\hbox{$\rightarrow$}\dimen0\ht\z@\multiply\dimen0 6\divide\dimen0 10\ht\z@\dimen0\box\z@}}
\def\myrightarrowfill@{\arrowfill@\relbar\relbar\myrightarrow}
\def\myleftarrow{{\setbox\z@\hbox{$\leftarrow$}\dimen0\ht\z@\multiply\dimen0 6\divide\dimen0 10\ht\z@\dimen0\box\z@}}
\def\myleftarrowfill@{\arrowfill@\myleftarrow\relbar\relbar}
\newcommand{\myxrightarrow}[2][]{\ext@arrow 0359\myrightarrowfill@{#1}{#2}}
\newcommand{\myxleftarrow}[2][]{\ext@arrow 3095\myleftarrowfill@{#1}{#2}}
\makeatother
\newcommand{\mtilde}{{\mathchoice
		{\widetilde{m}}
		{\widetilde{m}}
		{\rlap{$\scriptscriptstyle{m}$}\vphantom{\raise0pt\hbox{$m$}}\smash{\lower2.5pt\hbox{$\scriptscriptstyle\widetilde{\phantom{\scriptscriptstyle{m}}}$}}}
		{\rlap{$\scriptscriptstyle{m}$}\vphantom{\raise.2pt\hbox{$m$}}\smash{\lower2.05pt\hbox{$\scriptscriptstyle\widetilde{\phantom{\scriptscriptstyle{m}}}$}}}}}

\newcommand{\Mtilde}{{\mathchoice
		{\rlap{$M$}\mkern1mu\smash[b]{\lower.5pt\hbox{$\widetilde{\phantom{M}}$}}\mkern-1mu}
		{\rlap{$M$}\mkern1mu\smash[b]{\lower.5pt\hbox{$\widetilde{\phantom{M}}$}}\mkern-1mu}
		{\rlap{$\scriptstyle{M}$}\mkern1mu\smash[b]{\lower.5pt\hbox{$\widetilde{\phantom{\scriptstyle{M}}}$}}\mkern-1mu}
		{\widetilde{M}}}}

\newcommand{\A}{{\mathbf A}}

\ifdef{\C}%
{{\renewcommand{\A}{{\mathbb{A} }}}
\renewcommand{\C}{{\mathbf C}}}%
{\newcommand{\C}{{\mathbf C}}}

\newcommand{\cF}{\mathrm F}

\newcommand{\cFplus}{\mathrm F_+}
\newcommand{\cFconst}{\mathrm F_\const}

\newcommand{\const}{\mathrm{const}}

\newcommand{\Gal}{\mathrm{Gal}}

\newcommand{\Pic}{\mathrm{Pic}}
\newcommand{\Br}{\mathrm{Br}}

\newcommand{\Div}{\mathrm{Div}}
\renewcommand{\phi}{\varphi}
\renewcommand{\emptyset}{\varnothing}

\newcommand{\Hom}{{\mathrm{Hom}}}

\newcommand{\chapeau}{{\rlap{\smash{\hbox{\lower4pt\hbox{\hskip1pt$\widehat{\phantom{u}}$}}}}}}
\newcommand{\Picplushat}{\Pic_+^{{\smash{\hbox{\lower4pt\hbox{\hskip0.4pt$\widehat{\phantom{u}}$}}}}}}
\newcommand{\PicplusAhat}{\Pic_{+,\A}^{{\smash{\hbox{\lower4pt\hbox{\hskip.4pt$\widehat{\phantom{u}}$}}}}}}

\newcommand{\Pichat}{\Pic^{{\smash{\hbox{\lower4pt\hbox{\hskip0.4pt$\widehat{\phantom{u}}$}}}}}}

\ifdef{\G}%
{\renewcommand{\G}{{\mathcal{G}}}}%
{\newcommand{\G}{{\mathcal{G}}}}

\newcommand{\coker}{\mathrm{coker}}

\makeatletter
\renewcommand{\tocsection}[3]{%
	\indentlabel{\@ifnotempty{#2}{\bfseries\ignorespaces#1 #2\quad}}\bfseries#3}

\renewcommand{\tocsubsection}[3]{%
	\indentlabel{\@ifnotempty{#2}{\hspace{1.6em}\ignorespaces#1 #2\quad}}#3}
\makeatother

\makeatletter\let\@wraptoccontribs\wraptoccontribs\makeatother

\usepackage{threeparttable}
\usepackage{threeparttable, tablefootnote}

\newcommand\Deltabar{{\overline{\Delta}}}
\newcommand\sigmabar{{\overline{\sigma}}}

\newcommand\That{{\widehat{T}}}

\newcommand\ZZ{\mathbb{Z}}
\newcommand\QQ{\mathbb{Q}}
\newcommand\sm{\mathrm{sm}}
\usepackage{soul}
\begin{document}

\title[Arithmetic purity for toric varieties]
{Arithmetic purity of strong approximation for toric varieties with constant global sections}

\author{Dasheng Wei}

\address{Academy of Mathematics and System Science, CAS, Beijing 100190,
  P.\ R.\ China \emph{and} School of mathematical Sciences, University of  CAS, Beijing
  100049, P.\ R.\ China}

\email{dshwei@amss.ac.cn}

\author{Fei Xu}

\address{School of Mathematical Sciences, Capital Normal University, Beijing 100048, P. \ R. \ China}

\email{6091@cnu.edu.cn}

\author{Yi Zhu}

\address{Bethesda, MD, USA 20817}

\email{math.zhu@gmail.com}

\date{July 8, 2026}

\begin{abstract} Let $X$ be a smooth toric variety over a number field $k$ such that the global sections of $X$ are constant. We show that $U(k)$ is dense in $U(\A_k)^{\Br_1 (U)}$ for any open subset $U$ of $X$ such that the codimension of $X\setminus U$ in $X$ is at least 2.
\end{abstract}


% 14G05: Algebraic geometry->Arithmetic problems. Diophantine geometry->Rational points
%
% 11D57: Number theory->Diophantine equations->Multiplicative and norm form equations
%
% 14F22: Algebraic geometry->(Co)homology theory->Brauer groups of schemes

\maketitle

%\tableofcontents

\section{Introduction}

It is well-known that weak approximation of smooth varieties over a number field is birationally invariant. However, this is no longer true for strong approximation of smooth varieties. Based on various geometric phenomena, the following arithmetic purity of strong approximation can be expected, which was proposed by Wittenberg in \cite[Question 2.11]{witten2015}.

\begin{conjecture}
     If $X$ is a smooth variety over number field $k$ satisfying strong approximation off a finite set of primes $S$ of $k$, whether any open subset $U$ of $X$ over $k$ with ${\rm{codim}}(X\setminus U, X)\geq 2$ also satisfies strong approximation off $S$?
\end{conjecture}

The first example to support this conjecture is an affine space $\mathbb A^n$, which was independently verified by Cao and Xu in \cite[Proposition 3.4]{CaoXu2018} and Wei in \cite[Lemma 2.1]{weitorus}. Later, this conjecture has been proved for semi-simple simply connected quasi-split linear algebraic groups and the related homogeneous spaces by Cao, Liang and Xu in \cite{CaoLiangXu2019}. More general conjecture involved Brauer-Manin obstruction is also proposed in \cite{CaoLiangXu2019}. In \cite{CaoHuang2020}, Cao and Huang further extend the quasi-split case to the isotropic case. Since weak approximation is the same as strong approximation for proper smooth varieties, the arithmetic purity conjecture indicates that any open sub-varieties of the proper variety with weak approximation by removing a closed subset of codimesion $\geq 2$ should satisfy strong approximation off $\emptyset$. Indeed, Chen showed such the arithmetic purity of strong approximation for certain proper toric varieties in \cite{Chen24}. In this paper, we will extend this result to all proper toric varieties.

Another motivation of this paper is to understand the recent work \cite{sant} of Santens. Using harmonic analysis on universal torsors, he proved that a smooth toric variety with constant global sections satisfies strong approximation off $\emptyset$. By studying some geometric properties of toric varieties, we will show that a smooth toric variety with constant global sections satisfies the arithmetic purity of strong approximation off $\emptyset$, which is the main result of this paper.

\begin{thm}\label{main} Let $X$ be a smooth toric variety over a number field $k$ with $H^0(X, \mathcal O_X)=k$. If $U$ is an open subset of $X$ such that the codimension of $X\setminus U$ in $X$ is at least 2, then $U(k)$ is dense in $U(\A_k)^{\Br_1(U)}$.
\end{thm}

Notation and terminology are standard if not explained. Let $k$ be a field and $\bar k$ be a fixed algebraic closure of $k$. For a scheme $X$ over $k$, we write $X_{\bar k} =X\times_k \bar k$ and
$$ \Br(X)= H_{et}^2(X, \mathbb{G}_m) \ \ \ \text{and} \ \ \ \Br_1(X)= \ker (\Br(X) \longrightarrow \Br(X_{\bar k})) $$
and $$\Br_a(X)= \coker( \Br(k)\longrightarrow \Br_1(X)) $$ with respect to the structure morphism. A variety over $k$ means a separated and integral scheme of finite type over $k$. For a variety $X$ over $k$, we use $X^{\sm}$ to denote the smooth locus of $X$ over $k$. When $k$ is a number field, we write $\A_k$ for the adeles of $k$ and $X(\A_k)$ for adelic points of $X$. Let
$$ X(\A_k)^{B}= \{ (x_v)_v\in X(\A_k): \ \sum_{v} inv_v(\xi(x_v)) =0 \in \mathbb{Q}/\mathbb{Z} , \  \forall \xi\in B\}  $$ be a closed subset of $X(\A_k)$ with the adelic topology for any subgroup $B$ of $\Br(X)$. We say that a variety $X$ over a number field $k$ satisfies strong approximation with algebraic Brauer-Manin obstruction off $\emptyset$ if $X(k)$ is dense in $X(\A_k)^{\Br_1(X)}$ with adelic topology.

\section{Some results for toric varieties}
In this section, we provide some results for toric varieties over an arbitrary field, which we need in the next section. Some basic results and concepts about toric varieties have been summarized in \cite[\S 2]{Chen24}. First, we recall the following concept given in \cite[Definition 2.5]{CaoXu2018}.

\begin{definition} A toric variety $X$ over $k$ is called divisorial if all minimal orbits of the torus are divisors of $X$. Equivalently, the dimension of any cone in the fan of X is strictly less than 2.
\end{definition}

Since a normal toric variety $X$ over $k$ corresponds a fan $\Delta$ which is stable under the action of $\Gal(\bar k/k)$, the sub-fan of $\Delta$ consisting of all 1-dimensional cones and 0 which is also stable under the action of $\Gal(\bar k/k)$ corresponds the unique open toric divisorial subvariety $Y$ over $k$ such that $X\setminus Y$ has codimension at least 2 in $X$. It should be pointed out that a divisorial toric varieties is automatically smooth by \cite[Theorem 3.1.19]{CoxLS11}.

%The explicit description of toric varieties with constant global section is given by \cite[Exercise 4.3.4]{CoxLS11} or .

\begin{lem}\label{cpt} If $U$ is a smooth divisorial toric variety over $\bar k$ such that $H^0(U, \mathcal O_U) = \bar k$, then $U$ has a normal toric compactification $X$ such that $X\setminus U$ has codimension at least 2 in $X$.
\end{lem}
\begin{proof}  Let $T$ be the torus of $U$, $\hat{T}$ be the characters of $T$ and $N=\Hom_{\mathbb Z}(\hat T, \mathbb Z)$. Write $\Delta_U$ to be the associated fan of $U$.
Since $U$ is divisorial, the fan $\Delta_U$ consists of 1-dimensional cones (rays) and 0. Let $\{v_1, \cdots, v_d\}$ be the set of primitive generators of all rays in $\Delta_U$ and
$$ C= \mathbb{R}_{\geq 0} v_1 + \cdots + \mathbb{R}_{\geq 0} v_d \subseteq N \otimes_{\mathbb Z} \mathbb R $$ be a cone generated by $\{v_1, \cdots, v_d\}$. Since
$$\bar k= H^0(U\times_k \bar k, \mathcal O_{U_{\bar k}}) = \bigcap_{i=1}^d \bar k[\hat T \cap v_i^{\vee} ] = \bar k[\hat T\cap C^{\vee}]$$ by \cite[Theorem 1.2.18 and Theorem 3.1.7]{CoxLS11} (see also \cite[Lemma 2.1]{CaoXu2018}), we have $$\hat T\cap C^{\vee}=\{ \chi \in \hat T: \ v_i(\chi)\geq 0 \ \text{for} \  1\leq i\leq d \} =0 .$$
Then $C^{\vee}=0$. Therefore $ H^0(U, \mathcal O_U) = \bar k $ if and only if $C=N \otimes_{\mathbb Z} \mathbb R$.

\smallskip

Fixing a basis $\{e_1, \cdots, e_r\}$ of $N$ over $\mathbb Z$, we can define a norm
$$\parallel x \parallel = \sqrt{a_1^2+ \cdots + a_r^2};  \ \ \ \forall\ x=\sum_{i=1}^r a_i e_i \in  N\otimes_{\mathbb Z} \mathbb R $$ on $N\otimes_{\mathbb Z} \mathbb R$. Let $u_i= \frac{v_i}{\parallel v_i \parallel}$ for $1\leq i\leq d$ and write
$$ P= \{ \sum_{i=1}^d \lambda_i u_i: \ \lambda_i \geq 0 \ \ \text{and} \ \  \sum_{i=1}^d \lambda_i =1 \} \subset N\otimes_{\mathbb Z} \mathbb R . $$
Since $C=N \otimes_{\mathbb Z} \mathbb R$, the point $0$ is in the interior of $P$. Moreover, since the unit ball is strictly convex, the points $ u_1, \cdots, u_d $ are all vertices of $P$.

Define $\Delta$ to consist of all possible $\text{cone}(F)$
$$ \text{cone}(F)= \mathbb{R}_{\geq 0} v_{i_1} + \cdots + \mathbb{R}_{\geq 0} v_{i_k} $$ where $F=\{ u_{i_1}, \cdots, u_{i_k} \}$ is the vertices of a proper face of $P$.
Since each ray $\mathbb{R}_{\geq 0} u_i=\mathbb{R}_{\geq 0} v_i$ only intersects $P$ at $u_i$ for $1\leq i\leq d$, we have
$$ \text{cone}(F_1)\cap \text{cone}(F_2)= \text{cone}(F_1\cap F_2) .$$
This implies that $\Delta$ is a fan. Let $X$ be the toric variety associated to $\Delta$. Then $U$ is an open toric sub-variety. Since
$$|\Delta|= \bigcup_{F\in \Delta} F= C= N \otimes_{\mathbb Z} \mathbb R, $$ the variety $X$ is a toric compactification of $U$ by \cite[Theorem 3.4.6]{CoxLS11}.
Since all cones in $\Delta$ are strongly convex by \cite[Proposition 1.2.12]{CoxLS11}, the toric variety $X$ is normal by \cite[Theorem 1.3.5]{CoxLS11}. Moreover,
since $\Delta\setminus \Delta_U$ consists of cones of dimension larger than 1, we conclude that
$X\setminus U$ has codimension at least 2 in $X$ as required.
\end{proof}

The following example explains that Lemma \ref{cpt} is not true for a general smooth toric variety with constant global sections.

\begin{example} Consider the following vectors in $\mathbb Z^4$
$$ e_1=(1, 0, 0, 0), \ \ e_2=(0, 1, 0, 0), \ \ e_3=(0, 0, 1, 0), \ \ e_4=(0, 0, 0, 1)$$
$$ e_5= (0, -1, -1, -1), \ \ e_6= (-1, -1, 0, -1), \ \ e_7=(1, 2, 1, 1).$$
Let $\rho_i=\mathbb R_{\geq 0} e_i$ for $1\leq i\leq 7$ and
$$ \tau_1= \mathbb R_{\geq 0} e_2 + \mathbb R_{\geq 0} e_5 + \mathbb R_{\geq 0} e_7$$   $$\tau_2= \mathbb R_{\geq 0} e_2 +\mathbb R_{\geq 0} e_3 + \mathbb R_{\geq 0} e_7 $$ $$\tau_3= \mathbb R_{\geq 0} e_1 + \mathbb R_{\geq 0} e_5 + \mathbb R_{\geq 0} e_6  $$ be 3-dimensional cones.
Define $\Sigma$ to consist of $\tau_1$, $\tau_2$, $\tau_3$, all their faces and $\rho_4$. Then we can show that $\Sigma$ is a fan. Indeed, $\rho_4\cap \tau_1=\rho_4\cap \tau_2=\rho_4 \cap \tau_3= 0$ and $$\tau_1\cap \tau_2 = \mathbb R_{\geq 0} e_2+\mathbb R_{\geq 0} e_7 , \ \ \  \tau_1\cap \tau_3=\mathbb R_{\geq 0} e_5 \ \ \ \text{and} \ \ \ \tau_2\cap \tau_3= 0 $$ are the faces of $\tau_1$, $\tau_2$, $\tau_3$. The rest of cases follow from the straight verification. Let $U$ be the toric variety corresponding to the fan $\Sigma$.

\smallskip

(1) The toric variety $U$ is smooth and $H^0(U, \mathcal O_U) =k$.

Indeed, since
$$\det(e_2, e_5, e_7, e_3)=\det(e_2, e_3, e_7, e_4)=-\det(e_1, e_5, e_6, e_4)=1, $$
the minimal generators of non-zero cones is part of basis of $\mathbb Z^4$. Then $U$ is smooth by \cite[Theorem 1.3.12]{CoxLS11}. Since
$$ e_1+2e_2+e_3+3e_4+2e_5+2e_6+e_7=0 $$ and $e_1, e_2, e_3, e_4$ is a basis of $\mathbb R^4$, we have
$$\mathbb R^4= \mathbb R_{\geq 0} e_1 + \mathbb R_{\geq 0} e_2 + \mathbb R_{\geq 0} e_3 + \mathbb R_{\geq 0} e_4 + \mathbb R_{\geq 0} e_5 + \mathbb R_{\geq 0} e_6 + \mathbb R_{\geq 0} e_7 . $$ By the first part of proof of Proposition \ref{cpt}, we conclude $H^0(U, \mathcal O_U) =k$.

\smallskip
(2) The toric variety $U$ has no toric compactification $X$ over $k$ such that $X\setminus U$ has codimension at least 2 in $X$.

Assume, for contradiction, that there exists a complete fan $\Delta\supset \Sigma$ such that the 1-dimensional cones of $\Delta$ are $\rho_1, \cdots, \rho_7$. Since $\Delta$ is complete, we have
\begin{equation}\label{complete}
|\Delta|= \bigcup_{\tau\in \Delta} \tau = \mathbb{R}^4 \end{equation}
by \cite[Theorem 3.4.6]{CoxLS11}. This implies that $\Delta$ contains 4-dimensional cones.

Consider $\tau_1\in \Sigma\subset \Delta$ and the $\mathbb R$-subspace $spn(\tau_1) \subset \mathbb{R}^4$ generated by $\tau_1$. Since $$e_1=-e_2+e_5+e_7\in spn(\tau_1)\setminus \tau_1,$$ there is no cone containing $\tau_1$ and $\rho_1$ in $\Delta$. The rays of $\Delta$ outside $spn(\tau_1)$ are $\rho_3, \rho_4$ and $\rho_6$ . By computing the determinants
$$\det(e_2, e_5, e_7, e_3)=\det(e_2, e_5, e_7, e_6)=1 \ \ \ \text{and} \ \ \ \det(e_2, e_5, e_7, e_4)=-1 ,$$ the vectors $e_3$ and $e_6$ are in the same side of $spn(\tau_1)$ and $e_4$ is in the other side of $spn(\tau_1)$.
By (\ref{complete}), there must be a cone $\gamma$ containing $\tau_1$ and one of $\{\rho_3, \rho_6\}$ in $\Delta$.
  If $\gamma$ contains both $\rho_3$ and $\rho_6$, then $\gamma$ contains $\tau_2$. Since
$$e_6=e_2+e_3-e_7 \in spn(\tau_2)\setminus \tau_2 $$ where $spn(\tau_2)$ is the $\mathbb R$-subspace generated by $\tau_2$, the cone $\tau_2$ is not the face of $\gamma$. A contradiction is derived. Therefore
$$ \gamma= \text{cone}(e_2, e_3, e_5, e_7)=\mathbb R_{\geq 0} e_2 + \mathbb R_{\geq 0} e_3 + \mathbb R_{\geq 0} e_5 + \mathbb R_{\geq 0} e_7$$ or $$ \gamma= \text{cone}(e_2, e_5, e_6, e_7)=\mathbb R_{\geq 0} e_2 + \mathbb R_{\geq 0} e_5 + \mathbb R_{\geq 0} e_6 + \mathbb R_{\geq 0} e_7 .$$

Consider $\tau_2\in \Sigma\subset \Delta$ and the $\mathbb R$-subspace $spn(\tau_2) \subset \mathbb{R}^4$ generated by $\tau_2$. Since
$$ e_6 = e_2 + e_3 - e_7  \in span(\tau_2)\setminus \tau_2 , $$
there is no cones containing $\tau_2$ and $\rho_6$. The rays of $\Delta$ outside $spn(\tau_2)$ are $\rho_1$, $\rho_4$ and $\rho_5$. By computing the determinants
$$\det(e_2, e_3, e_7, e_4)=1 \ \ \ \text{and} \ \ \ \det(e_2, e_3, e_7, e_1)=\det(e_2, e_3, e_7, e_5)=-1 ,$$ the vectors $e_1$ and $e_5$ are in the same side of $spn(\tau_2)$ and $e_4$ is in the other side of $spn(\tau_2)$. Since $\rho_4$ is the only ray in this side of $spn(\tau_2)$, this forces
$$ \text{cone}(e_2, e_3, e_4, e_7)=\mathbb R_{\geq 0} e_2 + \mathbb R_{\geq 0} e_3 + \mathbb R_{\geq 0} e_4 + \mathbb R_{\geq 0} e_7 \in \Delta $$ by (\ref{complete}).

Consider $\tau_3\in \Sigma\subset \Delta$ and the $\mathbb R$-subspace $spn(\tau_3) \subset \mathbb{R}^4$ generated by $\tau_3$. Since
$$ e_3 = e_1 - e_5 + e_6  \in span(\tau_3) \setminus \tau_3 , $$
there is no cones containing $\tau_3$ and $\rho_3$. The rays of $\Delta$ outside $spn(\tau_3)$ are $\rho_2$, $\rho_4$ and $\rho_7$. By computing the determinants
$$\det(e_1, e_5, e_6, e_4)=-1 \ \ \ \text{and} \ \ \ \det(e_1, e_5, e_6, e_2)=\det(e_1, e_5, e_6, e_7)=1 ,$$ the vectors $e_2$ and $e_7$ are in the same side of $spn(\tau_3)$ and $e_4$ is in the other side of $spn(\tau_3)$. Since $\rho_4$ is the only ray in this side of $spn(\tau_3)$, this forces
$$ \text{cone}(e_1, e_4, e_5, e_6)=\mathbb R_{\geq 0} e_1 + \mathbb R_{\geq 0} e_4 + \mathbb R_{\geq 0} e_5 + \mathbb R_{\geq 0} e_6 \in \Delta $$ by (\ref{complete}). Now we distinguish the following two cases.

Case 1: $\text{cone}(e_2, e_3, e_5, e_7)\in \Delta$.

Since $\text{cone}(e_1, e_4, e_5, e_6) \in \Delta$, then $\text{cone}(e_2, e_3, e_5, e_7) \cap \text{cone}(e_1, e_4, e_5, e_6)$ is a common face. Since $\rho_5$ is the only common ray and
$$ e_3+e_5 = e_1+e_6 \in (\text{cone}(e_2, e_3, e_5, e_7) \cap \text{cone}(e_1, e_4, e_5, e_6)) \setminus \rho_5 , $$ a contradiction is derived.

Case 2: $\text{cone}(e_2, e_5, e_6, e_7)\in \Delta$.

Since $\text{cone}(e_2, e_3, e_4, e_7)\in \Delta$, then $\text{cone}(e_2, e_5, e_6, e_7)\cap \text{cone}(e_2, e_3, e_4, e_7)$ is a common face. Since the common rays are $\rho_2$ and $\rho_7$, this forces
$$\text{cone}(e_2, e_5, e_6, e_7)\cap \text{cone}(e_2, e_3, e_4, e_7)= \mathbb R_{\geq 0} e_2 + \mathbb R_{\geq 0} e_7 =\text{cone}(e_2, e_7) .$$
However
$$ e_6+e_7=e_2+e_3 \in (\text{cone}(e_2, e_5, e_6, e_7)\cap \text{cone}(e_2, e_3, e_4, e_7))\setminus \text{cone}(e_2, e_7) $$ and a contradiction is derived. \qed
\end{example}

\begin{lem}\label{lem:toric-codimension}
	Let $X$ and $X'$ be normal toric varieties over a field $k$ and $X' \xrightarrow{f} X$ be a surjective  toric morphism over $k$. Suppose that $\Delta'$ and $\Delta$ are the corresponding fans of $X'$ and $X$ respectively and $\Delta'\xrightarrow{g} \Delta$ is the corresponding map induced by $f$. Assume that $g$ maps all 1-dimensional cones to 1-dimensional cones or 0.

    If $Z\subset X $ has codimension at least $2$, then $f^{-1}(Z)\subset X' $ has codimension at least $2$.
\end{lem}
\begin{proof} Without loss of generality, we may assume that $k$ is algebraically closed.
Since $X'$ has the disjoint union of orbits
$$ X'= T'\cup (\bigcup_{\dim(\sigma')=1} O(\sigma')) \cup (\bigcup_{\dim(\sigma') \geq 2} O(\sigma'))$$
by \cite[Theorem 3.2.6]{CoxLS11} where $T'$ is the corresponding torus of $X'$ and $\sigma'$ are the cones in $\Delta'$, we only need to show that $$ \text{codim}(f^{-1}(Z)\cap T', T')\geq 2 \ \ \ \text{and} \ \ \ \text{codim}(f^{-1}(Z) \cap O(\sigma'), O(\sigma')) \geq 1$$ for $\dim(\sigma')=1$.

Indeed, since $f$ induces a surjective homomorphism $f|_{T'}: T'\rightarrow T$ of tori,
the morphism $f|_{T'}$ is smooth. This implies $ \text{codim}(f^{-1}(Z)\cap T', T')\geq 2$ as desired.

For any 1-dimensional cone $\sigma'\in \Delta'$, there is a 1-dimensional cone $\sigma\in \Delta$ such that
$$ f|_{O(\sigma')} : O(\sigma') \rightarrow O(\sigma) \ \ \ \text{or} \ \ \ f|_{O(\sigma')} : O(\sigma') \rightarrow T $$  by assumption, where $O(\sigma')$ and $O(\sigma)$ are the orbits of $T'$ and $T$ with the reduced structures corresponding to $\sigma'$ and $\sigma$ respectively. Since $f|_{O(\sigma')}$ is smooth and surjective, we obtain
$$ \text{codim}(f^{-1}(Z) \cap O(\sigma'), O(\sigma')) = \text{codim}(Z\cap O(\sigma), O(\sigma)) \geq 1 $$ or
$$ \text{codim}(f^{-1}(Z) \cap O(\sigma'), O(\sigma')) = \text{codim}(Z \cap T , T)\geq 2 $$ as desired. \end{proof}

It should be pointed out that $X^{\sm}\setminus Z$ is strong rationally connected by the main theorem of \cite{ChenSh11} when $X$ is a complete normal toric variety and $Z\subset X$ is a closed subset of codimension at least $2$.

\begin{prop} \label{weight-cover}
Let $X$ be a normal toric variety over $k$ with $H^0(X, \mathcal O_X)=k$. If $Y$ is the unique open toric divisorial subvariety of $X$ over $k$, there is a toric variety $U$ over $k$ with a surjective  toric morphism over $k$
    $$ f: U \longrightarrow Y $$ such that the following properties hold.

    (1) The associated torus $T'$ of $U$ satisfies $\widehat{T'}\cong \Div_{Y_{\bar k}\setminus T_{\bar k}} (Y_{\bar k})$ as $\Gal (\bar k/k)$-module.

    (2) The variety $U$ is the unique open toric divisorial subvarieties of a weighted projective space over $k$.

    (3) The closed subset $f^{-1}(Z)\subset U $ has codimension at least $2$ for any closed subset $Z$ in $Y$ with $\text{codim}(Z, Y)\geq 2$.

     (4) The morphism $f$ induces $f^*(\Div_{Y_{\bar k}\setminus T_{\bar k}} (Y_{\bar k}))$ as a direct summand of $\Div_{U_{\bar k}\setminus T'_{\bar k}} (U_{\bar k})$ as $\Gal(\bar k/k)$-module.

\end{prop}

\begin{proof}
 Let $T$ be the associated torus of $X$ and $\That$ be the group of characters of $T$ with the $\Gal(\bar k/k)$-action. Then the dual lattice $N=\Hom_\ZZ(\That,\ZZ)$ has the induced $\Gal(\bar k/k)$-action.
 Since
  $$H^0(Y_{\bar k}, \mathcal O_{Y_{\bar k}})=H^0(X_{\bar k}, \mathcal O_{X_{\bar k}})= \bar k$$
 by \cite[Theorem 6.45]{AGI2010},  there is a normal toric compactification $Y^c$ of $Y_{\bar k}$ over $\bar k$ by Lemma \ref{cpt}. Let $\Delta$ be the fan associated to $Y^c$ and $\Delta(1)$ be the set 1-dimensional cone. For each 1-dimensional cone $\rho$, we choose a primitive generator $\varrho\in N$.

By \cite[Corolary 3.8]{OdaP91} (see also \cite[\S 2. Definition]{OdaP91}), there is a refine simplicial fan $\Delta'$ such that  $$\Delta'(1)=\Delta(1) \ \ \ \text{and} \ \ \ \sigma=\bigcup_{\sigma'\in \sigma\cap \Delta'} \sigma' $$ for any $\sigma\in \Delta$.

Since $Y^c$ is complete, we have
 $$N\otimes_{\mathbb Z} \mathbb{R}= \bigcup_{\sigma\in \Delta} \sigma =\bigcup_{\sigma'\in \Delta'} \sigma' $$
by \cite[Theorem 3.4.6]{CoxLS11}.
 There is
the minimal cone $\sigmabar\in \Delta'$ containing $-\sum_{\rho\in \Delta(1)} \varrho$. Since $Y$ is defined over $k$, the set $\Delta(1)$ is stable under the action of $\Gal(\bar k/k)$. Therefore the vector $-\sum_{\rho\in \Delta(1)} \varrho$ is invariant under the action of $\Gal (\bar k/k)$. The minimality implies that $\sigmabar$ is also stable under the action of $\Gal (\bar k/k)$. Write
	\begin{equation}\label{eq:toric}
		-\sum_{\rho\in \Delta(1)} \varrho=\sum_{\rho'\in \sigmabar(1)}\alpha_{\rho'} \varrho'
	\end{equation}
	where $\alpha_{\rho'}\in \QQ_{>0}$. Since $\sigmabar$ is simplicial, the coefficients $\alpha_{\rho'}$ are uniquely determined. Since the left hand side of (\ref{eq:toric}) is stable under the action of $\Gal(\bar k/k)$, the right hand side of (\ref{eq:toric}) is also stable under the action of $\Gal (\bar k/k )$ by uniqueness.
Removing the denominators in (\ref{eq:toric}), we have the equation  	\begin{equation}\label{eq:toric3}
		\sum_{\rho\in \Deltabar(1)} n_{\rho}\varrho=0
	\end{equation}
	where $n_{\rho}\in \ZZ_{<0}$ are coprime for $\rho\in \Delta(1)$.

Let $N'$ be the free lattice with the basis $\{e_{\rho}: \ \rho\in \Delta(1)\}$. The action of $\Gal(\bar k/k)$ on $\Delta(1)$ induces the action of $\Gal(\bar k/k)$ on $N'$ by $\tau(e_{\rho})=e_{\tau(\rho)}$ for any $\tau\in \Gal(\bar k/k)$.
Write $$u_0=\sum_{\rho\in \Delta(1)} n_{\rho} e_{\rho} .$$ Then the vector $u_0$ is $\Gal(\bar k/k)$-invariant by (\ref{eq:toric}).

Let $\Sigma'$ be the fan  made up of $$\{ \mathbb R_{\geq 0}e_\rho: \ \rho\in \Delta(1)\} \cup \{ \mathbb R_{\geq 0} u_0\} \cup \{0\}$$ which is stable under the action of $\Gal(\bar k/k)$. Then the toric variety $U$ associated to $\Sigma'$ is defined over $k$ by Galois descent.

(1) Since the character group $\widehat{T'} \cong N'$ as $\Gal(\bar k/k)$-modules and $\Div_{Y_{\bar k}\setminus T_{\bar k}} (Y_{\bar k})$ is a free $\mathbb Z$-module generated by 1-dimensional cones in $\Delta(1)$, the property (1) follows from $$N'\cong \Div_{Y_{\bar k}\setminus T_{\bar k}} (Y_{\bar k})$$ as $\Gal (\bar k/k)$-modules.

(2) Since $U$ is smooth and $U(k)\neq \emptyset$, we conclude that $U$ is the unique open toric divisorial subvariety of a weighted projective space by \cite[Example 3.1.7 and Example 5.1.14]{CoxLS11} and \cite[Proposition 3.10]{Chen24}.

(3) Consider the $\mathbb Z$-linear map $$N'\longrightarrow  N; \ \ \ e_{\rho} \mapsto \varrho $$ which is compatible with the fans $\Sigma'$ and $\Delta(1)\cup \{0\}$ in the sense of \cite[Definition 3.3.1]{CoxLS11} and the actions of $\Gal(\bar k/k)$, it induces a toric morphism  $$f: U\longrightarrow Y$$ over $k$ by the Galois descent.
Since $f$ satisfies the condition of Lemma \ref{lem:toric-codimension}, we conclude that $f^{-1}(Z)\subset U $ has codimension at least $2$ for any closed subset $Z$ in $Y$ with $\text{codim}(Z, Y)\geq 2$.

(4) Since $$\Div_{U_{\bar k}\setminus T'_{\bar k}} (U_{\bar k}) \cong (\oplus_{\rho\in \Delta(1)}\mathbb Z e_{\rho})\oplus \mathbb Zu_0 $$ as $\Gal(\bar k/k)$-module, this isomorphism gives
$$ f^*(\Div_{Y_{\bar k}\setminus T_{\bar k}} (Y_{\bar k})) \cong \oplus_{\rho\in \Delta(1)}\mathbb Z e_{\rho}$$
and the property (4) follows as desired.
\end{proof}

\section{Strong approximation for toric varieties}\label{S:toric}

In this section, we will extend \cite[Theorem 4.1]{Chen24} to all proper normal toric varieties and give a proof of Theorem \ref{main}. For completeness' sake, we give a direct proof of \cite[Corollary 4.2]{Chen24}.

\begin{lemma}\label{lem:weighted}
	Let $X=\mathbb P(a_0,a_1,\cdots, a_n)$ be the weighted projective space over a number field $k$.  Let $Z\subset X$ be a closed subset of codimension at least $2$. Then $X^{\sm}\setminus Z$ satisfies strong approximation off $\emptyset$. 	
\end{lemma}
\begin{proof} By \cite[Lemma 3A.3]{BR86}, we assume $\text{gcd}(a_0, s_1, \cdots, a_n)=1$.  Then
 $$N =\ZZ^{n+1}/\ZZ\cdot(a_0,a_1,\cdots, a_n)$$ is a lattice of rank $n$. Let $\{u_0, u_1, \cdots, u_n\}$ be the images of the standard basis vectors $\{e_0, e_1, \cdots, e_n \}$ of $\ZZ^{n+1}$ in $N$. Let $\Sigma$ be a fan made up of all cones generated by proper subsets of $\{u_0,u_1,\cdots,u_n\}$. Then $X$ is the toric variety associated to $\Sigma$ by \cite[Example 3.1.17 and Example 5.1.14]{CoxLS11}. Let $\Delta$ be the sub-fan of $\Sigma$ generated by all 1-dimensional cones and 0. Then the associated toric variety $X_{\Delta}$ of $\Delta$ is the unique open toric divisorial subvariety of $X^{\sm}$ such that $X^{\sm}\setminus X_{\Delta}$ has codimension at least 2 in $X^{\sm}$.

Let $\{ e_{0, 1},\cdots, e_{0, a_0},\cdots, e_{n, 1},\cdots, e_{n, a_n} \}$
be the standard basis of $\ZZ^{a_0} \times \ZZ^{a_1}\times \cdots \times \ZZ^{a_n}$ and
$$N' =\ZZ^{a_0} \times \ZZ^{a_1}\times \cdots \times \ZZ^{a_n}/\ZZ \cdot (\sum_{i=0}^{n} \sum_{j=1}^{a_i} e_{ij}) . $$ Write $u_{ij}$  the image of $e_{i,j}$ in $N'$ and $\Delta'$  the fan made up of  $1$-dimensional cones generated by $u_{ij}$ for $0\leq i\leq n$ and $1\leq j\leq a_i$. Then the associated toric variety  $X_{\Delta'}$ of $\Delta'$ is the unique open toric divisorial subvariety of the projective space $\mathbb P^{d-1}_k$ with $d=a_0+a_1+\cdots+ a_n$ such that $\mathbb P^{d-1}_k\setminus X_{\Delta}$ has codimension at least 2 in $\mathbb P^{d-1}_k$ by \cite[Example 3.1.10]{CoxLS11}.
 	
Since the projection $p: N'\to N $ by sending $u_{ij}$ to $u_i$ is well-defined for $0\leq i\leq n$ and $1\leq j\leq a_i$ and is compatible with the fans $\Delta'$ and $\Delta$ in the sense of \cite[Definition 3.3.1]{CoxLS11}, it induces a surjective toric morphism $X_{\Delta'}\xrightarrow{f} X_\Delta$ over $k$ by \cite[Theorem 3.3.4]{CoxLS11}. Since $p$ is split, the induced homomorphism
$$ f|_{T'}: T' \longrightarrow T $$
of tori, where $T$ and $T'$ are  of $X_{\Delta}$ and $X_{\Delta'}$ respectively, is surjective and $\ker(f|_{T'})$ is also a split torus.
Since $p$ sends the 1-dimensional cones of $\Delta'$ to the 1-dimensional cones of $\Delta$, we obtain that $f^{-1}(X_{\Delta}\cap Z)$ has codimension at least 2 in $X_{\Delta'}$ by Lemma \ref{lem:toric-codimension}. Therefore  $$X_{\Delta'}\setminus f^{-1}(X_{\Delta}\cap Z) \cong \mathbb P^{d-1}_k\setminus Z'$$
where $Z'$ is a closed subset of codimension~$\geq 2$ in $\mathbb P^{d-1}_k$.

Consider an open subset  $$\emptyset \neq \prod_v W_v \subset (X_{\Delta}\setminus Z)(\A_k). $$
Since $T$ is open dense in $X_{\Delta}$,  we have $W_v\cap T(k_v)\neq \emptyset$. Moreover, since
$\ker(f|_{T'})$ is a split torus, then $T'(k_v)\xrightarrow{f} T(k_v)$ is surjective by Hilbert 90. This implies that $f^{-1}(W_v)\neq \emptyset$ for each prime $v$ of $k$. Therefore $$ \emptyset \neq \prod_v f^{-1}(W_v) \subset (X_{\Delta'}\setminus f^{-1}(X_{\Delta}\cap Z)) (\A_k) . $$
By \cite[Proposition 3.7]{Chen24} (see first part of the proof), we have
$$ (\prod_v f^{-1}(W_v)) \cap (X_{\Delta'}\setminus f^{-1}(X_{\Delta}\cap Z)) (k) \neq \emptyset . $$ This implies that $(\prod_v W_v) \cap (X_{\Delta}\setminus Z)(k) \neq \emptyset $ as desired.
\end{proof}

The following proposition is a variant of \cite[Proposition 2.3]{CaoLiangXu2019}.

\begin{prop} \label{dense} Let $X$ be a smooth and integral variety over a number field k with $\bar k[X]^\times = \bar k^\times$. Suppose that $\Pic(X_{\bar k})$ is finitely generated and $B$ is a subgroup of $\Br(X)$ such that $B\supset \Br_1(X)$ and $[B: \Br_1(X)]<\infty$. If $U$ is an open dense subset of $X$, then $U(\A_k)^B$ is dense in $X(\A_k)^B$.
\end{prop}

%\begin{proof} Assume that $X(\A_k)^B\neq \emptyset$. Since $B\supset \Br_1(X)$, there is a universal torsor $Y\xrightarrow{f} X$ under the multiplicative group $S$ such that the character group $\widehat{S} \cong \Pic(X_{\bar k})$ as $\Gal(\bar k/k)$-module by \cite[Theorem 3]{Sk99}. Let $V=U\times_X Y$ be the pull-back $Y$ to $U$ and fix a finite subgroup $\Lambda \subset \Br(X)$ such that $\Lambda \cdot \Br_1(X)=B$.

%Let $P_0$ be a suﬃciently large finite set of primes of $k$ including all Archimedean primes of $k$ such that
%
%(a)
%
%\end{proof}

\begin{proof} For any  $(x_v)_v\in X(\A_k)^B$, since $B\supset \Br_1(X)$, there is a universal torsor $Y\xrightarrow{f} X$ such that there is $(y_v)_v\in Y(\A_k)$ with $(f(y_v))_v=(x_v)_v$ by \cite[Theorem 3]{Sk99}. Let $V=U\times_X Y$ be the pull-back $Y$ to $U$ and fix a finite subgroup $\Lambda \subset \Br(X)$ such that $\Lambda \cdot \Br_1(X)=B$. Let $\Lambda'$ be the image of $\Lambda$ in $\Br(Y)$ by $f^*$. By the functority of Brauer groups, we have $(y_v)_v\in Y(\A_k)^{\Lambda'}$. Since $\Lambda'$ is finite, we can choose $(y'_v)_v\in V(\A_k)^{\Lambda'}$ such that $(y'_v)_v$ is very close to $(y_v)_v$, hence the image $(f(y'_v))_v\in U(\A_k)^B$ and $(f(y'_v))_v $ is very closed to $(x_v)_v$.
\end{proof}

Our main result Theorem \ref{main} follows from the following theorem, which also extends \cite[Theorem 4.1]{Chen24} to all complete toric varieties.

\begin{thm}\label{thm:toric}
	Let $X$ be a normal toric variety over $k$ with $H^0(X, \mathcal O_X)=k$. If $Z\subset X$ is a closed subset of codimension at least $2$,
	then $X^{\sm}\setminus Z$ satisfies strong approximation with algebraic Brauer-Manin obstruction off $\emptyset$.
\end{thm}
\begin{proof} Let $Y$ be the unique open toric divisorial subvariety of $X$. Then $Y\subset X^{\sm}$ and $\text{codim}(X^{\sm}\setminus Y, X^{\sm})\geq 2$.
Therefore $$\Br_1(X^{\sm})\cong \Br_1(Y) \ \ \ \text{and} \ \ \ H^0(Y, \mathcal O_Y)=H^0(X^{\sm}, \mathcal O_{X^{\sm}})=H^0(X, \mathcal O_X)=k$$  by \cite[Theorem 3.7.6]{ctskbook} and \cite[Theorem 6.45]{AGI2010} respectively. Since $\Pic(X_{\bar k}^{\sm})$ is finitely generated by \cite[Theorem 4.2.1]{CoxLS11}, we only need to show that $Y\setminus Z$ satisfies strong approximation with algebraic Brauer-Manin obstruction by Proposition \ref{dense}.

By Proposition \ref{weight-cover}, there is a toric variety $U$ over $k$ with a surjective  toric morphism over $k$
    $$ f: U \longrightarrow Y $$ satisfies the listed properties. Then
\begin{equation} \label{br} \xymatrix{
    0 \ar[r] & \Br_1(Y) \ar[r] \ar[d]_{f^*}  &  \Br_1(T) \ar[r] \ar[d]^{f^*}  & H^2(k, \Div_{Y_{\bar k}\setminus T_{\bar k}} (Y_{\bar k})) \ar[d]^{f^*} \\
   0 \ar[r] & \Br_1(U) \ar[r] & \Br_1(T') \ar[r] & H^2(k, \Div_{U_{\bar k}\setminus T'_{\bar k}} (U_{\bar k})) } \end{equation}
by \cite[Lemme 6.1 (6.1.3)]{Sansuc81} and the functoriality, where $T$ and $T'$ are the associated tori of $Y$ and $U$ respectively. Since $\bar k[U]^\times=\bar k^\times$ and  $\Pic(U_{\bar k})= \mathbb Z$ by Proposition \ref{weight-cover} (2) and \cite[Theorem 7.1]{BR86}, we have $\Br_1(U)=\Br(k)$ by \cite[Lemme 6.3 (iii)]{Sansuc81}.
Since the homomorphism
$$  H^2(k, \Div_{Y_{\bar k}\setminus T_{\bar k}} (Y_{\bar k})) \xrightarrow{f^*} H^2(k, \Div_{U_{\bar k}\setminus T'_{\bar k}} (U_{\bar k})) $$ is injective by Proposition \ref{weight-cover}  (4), we obtain
\begin{equation}\label{relative} 0\longrightarrow \Br_1(Y) \longrightarrow \Br_1(T) \xrightarrow{f^*} \Br_a(T') \end{equation}
by (\ref{br}).

Consider any open subset $\prod_v W_v$ of $(Y\setminus Z)(\A_k)$ satisfying
$$(\prod_v W_v) \cap (Y\setminus Z)(\A_k)^{\Br_1(Y)} \neq \emptyset .$$ Since $(T\setminus Z)(\A_k)^{\Br_1(Y)}$ is dense in $(Y\setminus Z)(\A_k)^{\Br_1(Y)}$ by Proposition \ref{dense}, we have
$$ (\prod_v W_v) \cap T(\A_k)^{\ker f^*} \neq \emptyset$$ by (\ref{relative}). Since $\widehat{T'}$ is a permutation $\Gal(\bar k/k)$-module by Lemma \ref{weight-cover} (1), we have $\Sha^1 (k, T')=0$. Then there are $y\in T(k)$ and $(x_v)_v\in T'(\A_k)$ such that $$f((x_v)_v) \cdot y \in \prod_v W_v $$ by \cite[Proposition 3.4]{CaoXu2018}. Therefore
$$  f^{-1}(y^{-1} \cdot \prod_v W_v) \neq \emptyset  \ \ \ \text{in} \ \ \ (U\setminus f^{-1}(y^{-1}\cdot Z))(\A_k). $$ Since $f^{-1}(y^{-1}\cdot Z)$ has codimension 2 in $U$ by Lemma \ref{lem:toric-codimension}, then $U\setminus f^{-1}(y^{-1}\cdot Z)$ satisfies strong approximation off $\emptyset$ by Lemma \ref{lem:weighted}. There is $x\in U(k)$ such that
$$ f(x) \in y^{-1} \cdot \prod_v W_v  \ \ \  \Longleftrightarrow \ \ \  y \cdot f(x) \in \prod_v W_v $$
as desired.
\end{proof}

\noindent\textbf{Acknowledgements.} This work grew out of participation in the group-research program at Tianyuan Mathematical Research Center in Yunnan. We are grateful to the Center for offering us an excellent research environment.
We would like to thank Yang Cao for drawing our attention to \cite{sant} and some useful comments in the original version of the paper.  The first named author is supported by National Key R\&D Program of China (Grant No. 2020YFA0712600) and NSFC (Grant Nos. 12371014 and 12231009) and the second named author is supported by National Key R\&D Program of China (Grant No. 2023YFA1009702) and NSFC (Grant No. 12231009).

\bibliography{myref-2}	
\bibliographystyle{alpha}	
\end{document}

_